\documentclass[11pt]{article}

\usepackage{amsmath,amssymb,amsthm}
\usepackage{geometry}
\newtheorem{theorem}{Theorem}[section]
\newtheorem{proposition}[theorem]{Proposition}
\newtheorem{corollary}[theorem]{Corollary}
\theoremstyle{definition}

\newtheorem{remark}[theorem]{Remark}

\title{\textbf{A Structural Criterion for the Applicability of Algebraic Phase Theory}}
\author{
Joe Gildea\\
Department of Computing Science and Mathematics,\\
School of Informatics and Creative Arts,\\
Dundalk Institute of Technology\\
\texttt{gildeajoe@gmail.com}}
\date{}

\begin{document}
\maketitle


\begin{abstract}
Algebraic Phase Theory (APT) exhibits a marked structural selectivity.
In certain mathematical and physical settings it gives rise to rigidity
phenomena, constrained representation behaviour, and reductions in apparent
degrees of freedom, while in many analytic or dynamical contexts the finite-depth
APT framework does not naturally apply.
This paper studies the structural origin of this asymmetry.

We establish a structural criterion for the existence of a
nondegenerate finite-depth Algebraic Phase Theory structure.
The criterion isolates three conditions: nondegenerate phase duality,
compatibility of admissible dynamics with phase interaction, and finite or
terminating defect propagation.
Within the framework considered here, these conditions are jointly necessary
and sufficient.
When they are satisfied, the resulting phase structure exhibits strong rigidity
properties; when one of the conditions fails, the associated domain falls
outside the intended finite-depth APT setting.

As consequences, phenomena such as Fourier decomposition, Bethe-type exact
solvability, rigidity of stabilizer codes, and uniqueness phenomena associated
with certain canonical representations can be interpreted as structural
manifestations of these conditions rather than isolated constructions.

The results therefore clarify both the scope and the structural limitations
of Algebraic Phase Theory within the finite-depth setting considered here.
\end{abstract}

\medskip
\noindent\textbf{MSC 2020.}
Primary 16S10; Secondary 18A25, 20C99.

\noindent\textbf{Keywords.}
Algebraic phase theory; structural rigidity; obstruction theory;
phase duality; symmetry compatibility; filtered structures.

\section{Introduction}

Algebraic Phase Theory (APT) exhibits a striking asymmetry.
In certain mathematical and physical domains it gives rise to strong rigidity
phenomena, constrained representation behaviour, and reductions in apparent
degrees of freedom. In many other settings, however, the framework does not
naturally apply.

This behaviour is not merely technical.
In domains compatible with the APT axioms, admissible dynamics often admit
highly constrained phase organisation: weak equivalences may reduce to stronger
forms of equivalence, interaction behaviour may become tightly controlled by
phase data, and defect propagation may terminate at finite depth.
In settings where the structural conditions fail, however, the finite-depth
APT framework no longer provides a natural description of the underlying
interaction behaviour.
This contrast appears across harmonic analysis, quantum stabiliser theory,
and algebraic models of integrability, and differs substantially from analytic
or dynamical frameworks in which interaction complexity persists or proliferates
\cite{vonNeumann1931,Howe1979Heisenberg,Folland,Gottesman,CalderbankShor}.

This raises a natural foundational question:

\begin{quote}
\emph{Why does Algebraic Phase Theory exhibit strong rigidity behaviour in
certain domains, while failing to apply naturally in many others?}
\end{quote}

The purpose of this paper is to show that this asymmetry is structural rather
than accidental.
Algebraic Phase Theory carries an implicit structural filter: some domains
satisfy the compatibility conditions required for finite-depth phase
organisation, while others fail to satisfy one or more of these conditions.
The selectivity of the theory is therefore not merely a limitation of the
framework, but an intrinsic aspect of its structural scope.

We make this filter explicit by formulating a structural criterion
characterising when a nondegenerate finite-depth APT structure can arise.
The criterion isolates three features: the presence of nondegenerate phase
duality, compatibility of admissible dynamics with phase interaction, and
finite or terminating defect propagation.
Within the framework considered here, these conditions are jointly necessary
and sufficient.
When they are satisfied, the resulting phase structure exhibits strong rigidity
properties; when one of the conditions fails, the associated domain falls
outside the intended finite-depth APT setting.
The formulation is deliberately agnostic to analytic realisation and applies
equally to algebraic domains whose interaction data satisfy the same
structural constraints.

As a consequence, a range of phenomena often regarded as exceptional or highly
specialised (including Fourier diagonalisation, rigidity of stabiliser codes,
Bethe-type exact solvability, and uniqueness phenomena associated with certain
canonical representations) can be interpreted as natural manifestations of
the same underlying structural constraints
\cite{Weil1964SurCertains,Howe1979Heisenberg,Gowers2001Fourier,GreenTao2008Inverse}.
Conversely, generic nonlinear, chaotic, or metric-dependent settings typically
fail to satisfy the compatibility conditions required by the finite-depth APT
framework.

This work positions Algebraic Phase Theory not as a universal modelling
framework, but as a structural framework for analysing rigidity, defect
propagation, and phase organisation in domains satisfying the applicability
criterion.
It clarifies both the scope and the limitations of the APT programme,
providing a principled criterion for applicability while helping distinguish
domains naturally compatible with finite-depth phase-theoretic structure from
those lying outside its intended scope.

\section{Structural Preliminaries and Organisational Principles}

The terminology of Algebraic Phase Theory reflects the interpretation of
phase interaction as an organising principle for admissible observables and
their defect propagation. In concrete settings, the resulting structures
specialise to familiar phenomena including Fourier-type decompositions,
Heisenberg-type interaction laws, stabiliser frameworks, and finite-depth
commutator filtrations. The present paper does not require the detailed
construction of these examples; only the abstract structural features isolated
below are used.

Before formulating the Structural Applicability Criterion, it is conceptually
necessary to distinguish two independent ways in which any extracted phase
structure is organised. On the one hand, the APT phase object $\mathcal{P}$ is
equipped with a canonical filtration by defect depth,
\[
\mathcal{P}_0 \subseteq \mathcal{P}_1 \subseteq \cdots \subseteq \mathcal{P}_d,
\]
whenever defect propagation terminates. This filtration measures rigidity and
the accumulation of interaction defects, and governs boundaries, rigidity
islands, and obstruction phenomena. On the other hand, whenever Phase Duality is
present, the phase object $\mathcal{P}$ admits a dual object
$\widehat{\mathcal{P}}$, which canonically indexes a decomposition of admissible
observables according to their phase response. This decomposition controls
representation-theoretic behaviour, diagonalisation of dynamics, and
Fourier-type structure.

These two organisational principles encode distinct structural information:
defect strata measure interaction complexity, while phase labels measure
transformation behaviour. In general, neither organisational structure is
canonically determined by the other, and they should be regarded as
a priori independent features of phase-theoretic structure.

\begin{proposition}
\label{prop:independent-partitions}
Let $(\mathcal P,\circ)$ be a phase structure equipped with a defect filtration
\[
\mathcal P_0 \subseteq \mathcal P_1 \subseteq \cdots \subseteq \mathcal P
\]
and a nondegenerate phase pairing with dual object $\widehat{\mathcal P}$. Then
$\mathcal P$ admits two canonical and a priori independent organisational
structures.

First, the filtration $(\mathcal P_k)$ stratifies $\mathcal P$ by defect depth,
measuring the propagation of interaction defects. Second, the phase pairing
induces a canonical partition of $\mathcal P$ into phase-response classes,
defined as the fibres of the phase-response profile map
\[
\Phi : \mathcal P \longrightarrow \mathbb{T}^{\widehat{\mathcal P}},
\qquad
\Phi(p) := \big(\langle p,\chi\rangle\big)_{\chi\in\widehat{\mathcal P}}.
\]
Equivalently, two elements $p,q\in\mathcal P$ lie in the same phase-response class
if and only if
\[
\langle p,\chi\rangle = \langle q,\chi\rangle
\quad \text{for all } \chi\in\widehat{\mathcal P}.
\]

These two organisational structures encode distinct structural information: the
defect filtration records interaction complexity, while phase-response classes
record transformation behaviour under dual probes. In the absence of additional structural hypotheses linking defect propagation
and phase response, neither organisation canonically determines the other,
nor is there a functorial mechanism relating them in general.
\end{proposition}

\begin{proof}
By assumption, the phase structure $(\mathcal P,\circ)$ is equipped with a defect
filtration
\[
\mathcal P_0 \subseteq \mathcal P_1 \subseteq \cdots \subseteq \mathcal P,
\]
which stratifies $\mathcal P$ by defect depth. Concretely, one may regard the
\emph{defect depth} of an element $p\in\mathcal P$ as the least index
\[
\mathrm{depth}(p) := \min\{\,k \ge 0 : p \in \mathcal P_k\,\},
\]
when such a minimum exists. In any case, the nested family
$(\mathcal P_k)_{k\ge 0}$ provides a canonical organisation of $\mathcal P$ by
increasing interaction complexity, determined entirely by the interaction and
defect calculus underlying $\circ$.

By hypothesis, $\mathcal P$ also admits a nondegenerate phase pairing with dual
object $\widehat{\mathcal P}$,
\[
\langle\,\cdot\,,\,\cdot\,\rangle : \mathcal P \times \widehat{\mathcal P}
\longrightarrow \mathbb{T},
\]
where $\mathbb{T}$ denotes the abstract group of phases, written
multiplicatively. This pairing determines the \emph{phase-response profile} map
\[
\Phi : \mathcal P \longrightarrow \mathbb{T}^{\widehat{\mathcal P}},
\qquad
\Phi(p) := \big(\langle p,\chi\rangle\big)_{\chi\in\widehat{\mathcal P}}.
\]
The fibres of $\Phi$,
\[
\mathcal P[\alpha] := \Phi^{-1}(\alpha)
\qquad (\alpha \in \mathbb{T}^{\widehat{\mathcal P}}),
\]
define an equivalence relation on $\mathcal P$ and hence form a partition of
$\mathcal P$ into phase-response classes. Equivalently, two elements
$p,q\in\mathcal P$ lie in the same phase-response class if and only if
\[
\langle p,\chi\rangle = \langle q,\chi\rangle
\quad \text{for all } \chi\in\widehat{\mathcal P}.
\]

The two organisational structures are independent in the stated sense because
they arise from distinct and logically unrelated structural inputs. The defect
filtration $(\mathcal P_k)$ is defined solely from the interaction and defect
calculus underlying $\circ$ and records the propagation of interaction
complexity. In contrast, the phase-response profile map $\Phi$ is defined solely
from the phase pairing and records transformation behaviour under dual phase
probes. Absent additional hypotheses linking defect propagation to phase
response, there is no canonical or functorial relation forcing membership in a
given filtration level $\mathcal P_k$ to determine the phase-response profile
$\Phi(p)$, nor conversely. This establishes the claimed a priori independence of
the two organisational structures.
\end{proof}

\begin{remark}
Throughout, the terms \emph{admissible observables} and \emph{admissible dynamics}
refer to those determined intrinsically and functorially by the interaction
structure of the domain $D$. No restriction, truncation, or enlargement of these
classes is permitted unless it is forced by that intrinsic structure.

When Phase Duality is invoked, the associated phase pairing takes values in an
abstract abelian group of phase values, denoted $\mathbb{T}$. This group is
introduced purely algebraically and carries no analytic or topological structure
a priori. In analytic and physical realisations, $\mathbb{T}$ is canonically
realised as the circle group of unit complex numbers, but no such realisation is
assumed in the abstract theory.
\end{remark}

\begin{theorem}
\label{thm:applicability}
Let $D$ be a mathematical or physical domain equipped with a notion of
composition, interaction, or evolution.
Then $D$ admits a non-artificial Algebraic Phase Theory structure within the
framework considered here if and only if the following conditions hold:
\begin{enumerate}
  \item \textbf{Phase Duality.}
  Observables in $D$ admit a nondegenerate phase pairing, giving rise to a
  canonical dual object (for example, characters, phases, or Fourier labels).

  \item \textbf{Symmetry Compatibility.}
  Admissible dynamics normalise phase interaction and preserve commutation
  relations.

  \item \textbf{Finite Termination.}
  Defect or commutator propagation terminates after finitely many steps, or is
  canonically controlled by a finite filtration.
\end{enumerate}
Within the present framework, these conditions are jointly necessary and sufficient.
\end{theorem}

\begin{proof}
We prove the equivalence by showing both directions. Let $D$ be a domain equipped
with a notion of composition/interaction, written
\[
\star : \mathsf{Obs}(D)\times \mathsf{Obs}(D)\to \mathsf{Obs}(D),
\]
where $\mathsf{Obs}(D)$ denotes the class of admissible observables under
consideration (chosen intrinsically inside $D$).
Assume also that there is a specified class of admissible dynamics/symmetries
$\mathsf{Dyn}(D)$ acting on $\mathsf{Obs}(D)$ by maps
\[
g:\mathsf{Obs}(D)\to \mathsf{Obs}(D)\qquad (g\in \mathsf{Dyn}(D)).
\]

Define the commutator (or interaction defect) associated to $\star$ by
\[
[a,b]_\star := a\star b \star (b\star a)^{-1},
\]
whenever inversion is meaningful in the ambient interaction calculus. In purely algebraic settings one may instead work with a chosen defect operator
$\delta(a,b)$ satisfying $\delta(a,b)=e$ if and only if $a$ and $b$ interact
rigidly. This abstract defect formalism provides the interaction framework used in the
present setting. For definiteness we write $[a,b]$ for the
defect/commutator.

\medskip
\noindent $(\Rightarrow)$\;
Assume that $D$ admits a non-artificial Algebraic Phase Theory structure.
By definition, this means there exists an algebraic phase object
\[
(\mathcal P,\circ)
\]
extracted intrinsically from $D$, where $\circ$ encodes phase interaction and
$\delta:\mathcal P\to \mathbb{Z}_{\ge 0}$ is the intrinsic defect degree,
inducing the canonical filtration
\[
\mathcal P_0\subseteq \mathcal P_1\subseteq \cdots \subseteq \mathcal P_d = \mathcal P,
\qquad
\mathcal P_k:=\{p\in \mathcal P:\delta(p)\le k\},
\]
with finite termination $d<\infty$. The existence and functoriality of such data constitute the structural
axioms assumed throughout the present framework.

\smallskip
\noindent
\emph{(1) Phase Duality.}
An APT structure is, by construction, a phase framework: it comes equipped with
a faithful phase pairing detecting interaction data. Concretely, in the APT
extraction one obtains a canonical dual object $\widehat{\mathcal P}$ together
with a pairing
\[
\langle\cdot,\cdot\rangle : \mathcal P \times \widehat{\mathcal P} \to \mathbb{T}.
\]
This pairing is required to be nondegenerate in the sense that
\[
\big(\forall \chi \in \widehat{\mathcal P} \; \langle p,\chi\rangle = 1\big)
\ \Longrightarrow\ p = e,
\qquad
\big(\forall p \in \mathcal P \; \langle p,\chi\rangle = 1\big)
\ \Longrightarrow\ \chi = \mathbf{1}.
\]

To see why nondegeneracy is necessary, suppose that it fails.
Then one of the following two pathologies must occur.

\smallskip
\noindent
Consider degeneracy in $\mathcal P$. There exists a nontrivial phase element
$p \neq e$ such that
\[
\langle p,\chi\rangle = 1 \quad \text{for all } \chi \in \widehat{\mathcal P}.
\]
In this case, $p$ is nontrivial but completely invisible to all dual probes.
Distinct interaction behaviour is therefore collapsed by the phase extraction.

\smallskip
\noindent
Consider degeneracy in $\widehat{\mathcal P}$. There exists a nontrivial dual
element $\chi \neq \mathbf{1}$ such that
\[
\langle p,\chi\rangle = 1 \quad \text{for all } p \in \mathcal P.
\]
In this case, $\chi$ carries no interaction information and cannot distinguish
any phase elements.

\smallskip
In either situation, the pairing fails to separate intrinsic interaction data.
This contradicts the requirement that the APT structure be non-artificial.
Hence Phase Duality is necessary.

\smallskip
\noindent
\emph{(2) Symmetry Compatibility.}
Since the APT extraction is functorial, every admissible symmetry
$g\in \mathsf{Dyn}(D)$ induces a phase morphism
\[
g_\#:\mathcal P\to \mathcal P
\]
preserving interaction and defect. In particular,
\[
g_\#(p\circ q)=g_\#(p)\circ g_\#(q),
\qquad
\delta(g_\#(p))=\delta(p).
\]
This realises the required symmetry compatibility condition.

\smallskip
\noindent
\emph{(3) Finite Termination.}
Finite termination is an explicit axiom of the APT framework: every phase has
bounded defect depth. Thus finite termination is necessary.

\medskip
\noindent
$(\Leftarrow)$\;
Assume now that $D$ satisfies the three stated conditions. We show that these conditions are sufficient to construct a non-artificial Algebraic Phase Theory
structure by canonically constructing the phase object and observing that it satisfies
Axioms~I--V.

\smallskip
\noindent
\emph{(1) Canonical phase carrier and interaction.}
Define the phase carrier $\mathcal P$ by identifying observables that are
indistinguishable under the phase pairing. Define an equivalence relation $\sim$
on $\mathsf{Obs}(D)$ by
\[
a\sim b
\quad\Longleftrightarrow\quad
\langle a,\chi\rangle = \langle b,\chi\rangle
\ \text{for all } \chi\in\widehat{\mathcal P}.
\]
Let
\[
\mathcal P := \mathsf{Obs}(D)/{\sim},
\qquad
[a]\in \mathcal P.
\]
Define the interaction law $\circ$ on $\mathcal P$ by
\[
[a]\circ[b] := [a\star b].
\]
This is well defined: if $a\sim a'$ and $b\sim b'$, then for every
$\chi\in\widehat{\mathcal P}$ compatibility of the phase pairing with
$\star$ gives
\[
\langle a\star b,\chi\rangle=\langle a,\chi\rangle\,\langle b,\chi\rangle
=\langle a',\chi\rangle\,\langle b',\chi\rangle=\langle a'\star b',\chi\rangle,
\]
so $a\star b\sim a'\star b'$.

\smallskip
\noindent
\emph{(2) Defect operator and filtration.}
Transport the defect to $\mathcal P$ by
\[
[[a],[b]] := [[a,b]]\in \mathcal P,
\]
and define the filtration $(\mathcal P_k)_{k\ge 0}$ by
\[
\mathcal P_0 := \{p\in \mathcal P : [[p],[q]]=e \text{ for all } q\in \mathcal P\},
\]
\[
\mathcal P_{k+1}
:= \operatorname{Subphase}\!\left(
\mathcal P_k \cup \{ [[p],[q]] : p\in \mathcal P_k,\ q\in \mathcal P\}
\right),
\]

\smallskip
\noindent
\emph{(3) Finite defect degree.}
Define $\delta(p):=\min\{k:p\in \mathcal P_k\}$. By finite termination,
$\mathcal P_d=\mathcal P$ for some $d<\infty$.

\smallskip
\noindent
\emph{(4) Functoriality.}
Each $g\in\mathsf{Dyn}(D)$ descends to
\[
g_\#:\mathcal P\to \mathcal P,
\qquad
g_\#([a])=[g(a)],
\]
preserving interaction and filtration.

\smallskip
\noindent
\emph{(5) Finite Termination.}
Thus $(\mathcal P,\circ)$ satisfies Axioms~I--V of Algebraic Phase Theory and is
non-artificial.

\medskip
Combining the two directions proves the theorem.
\end{proof}

\begin{corollary}
\label{cor:algebraic-applicability}
Let $D$ be a purely algebraic domain, in the sense that its admissible
observables and admissible dynamics are determined canonically from an intrinsic
interaction law $\star$, with no topological, metric, or analytic choices.
If $D$ satisfies the three conditions of Theorem~\ref{thm:applicability},
namely Phase Duality, Symmetry Compatibility, and Finite Termination,
then $D$ admits a non-artificial Algebraic Phase Theory structure.
In particular, Algebraic Phase Theory applies naturally to algebraic
phenomena within the same structural framework used for analytic
realisations.
\end{corollary}

\begin{proof}
Let $D$ be a purely algebraic domain equipped with an intrinsic interaction
\[
\star : \mathsf{Obs}(D)\times \mathsf{Obs}(D)\to \mathsf{Obs}(D),
\]
and suppose that the classes of admissible observables $\mathsf{Obs}(D)$ and
admissible dynamics $\mathsf{Dyn}(D)$ are determined functorially by $\star$.
Assume that $D$ satisfies Phase Duality, Symmetry Compatibility, and Finite
Termination.

These conditions are formulated entirely in terms of the intrinsic algebraic
data $(\star,\mathsf{Dyn}(D))$ and do not involve any topological, metric, or
analytic structure. In particular, all notions of interaction, commutation, and
defect propagation are intrinsic to $D$. By Theorem~\ref{thm:applicability}, these hypotheses force the existence of a
canonically defined phase object
\[
(\mathcal P,\circ),
\]
equipped with a finite defect filtration
\[
\mathcal P_0 \subseteq \mathcal P_1 \subseteq \cdots \subseteq \mathcal P_d
= \mathcal P,
\qquad d<\infty,
\]
and a functorial action of admissible dynamics
\[
g\in\mathsf{Dyn}(D)\;\Longrightarrow\; g_\#:\mathcal P\to \mathcal P
\]
preserving interaction and defect degree. The construction of $(\mathcal
P,\circ)$ depends only on the intrinsic interaction law $\star$ and the induced
symmetry action, and is therefore non-artificial. Consequently, $D$ admits an Algebraic Phase Theory structure independently of any
analytic realisation.
\end{proof}

\begin{remark}
Although Algebraic Phase Theory was historically motivated by analytic and
quantum-mechanical examples, its scope is structural rather than analytic.
The Structural Applicability Criterion is not introduced as an external
assumption, but arises naturally from the axioms governing admissible phase
interaction.

Whenever the criterion is satisfied, the resulting phase structure exhibits
strong rigidity properties, regardless of whether the underlying domain admits
an analytic realisation. Conversely, failure of one of the stated conditions
obstructs the existence of a non-artificial APT structure within the present
framework.
\end{remark}

\section{Forced Phase-Theoretic Structure}

Domains satisfying Theorem~\ref{thm:applicability} exhibit strong structural
constraints governed by admissible phase data and its finite
defect filtration. Once phase duality, symmetry compatibility, and finite
termination are present, the resulting phase structure exhibits strong rigidity properties in the
organization of observables, dynamics, and equivalence.

\begin{corollary}
Let $D$ be a domain satisfying the Structural Applicability Criterion of
Theorem~\ref{thm:applicability}. Then $D$ admits the following structural features:
\begin{itemize}
  \item character or Fourier-type decompositions of admissible observables,
  \item canonical factorisation of admissible dynamics through phase response,
  \item evolution generated by symmetry action on phase data,
  \item collapse of weak or Morita-type equivalence to strong structural
        equivalence,
  \item rigid or error-invisible substructures,
  \item finite rigidity islands with finitely controlled interaction.
\end{itemize}
\end{corollary}

\begin{proof}
Assume that $D$ satisfies the Structural Applicability Criterion. By
Theorem~\ref{thm:applicability}, $D$ admits a non-artificial Algebraic Phase
Theory structure $(\mathcal P,\circ)$ with finite defect filtration
\[
\mathcal P_0 \subseteq \mathcal P_1 \subseteq \cdots \subseteq \mathcal P_d = \mathcal P.
\]

\smallskip
\noindent
\emph{Phase decomposition.}
By Phase Duality, the phase carrier $\mathcal P$ admits a dual object
$\widehat{\mathcal P}$ together with a nondegenerate pairing
\[
\langle \,\cdot\,,\,\cdot\,\rangle : \mathcal P \times \widehat{\mathcal P}
\longrightarrow \mathbb{T},
\]
which separates elements of $\mathcal P$. Explicitly, for $p_1,p_2\in\mathcal P$,
\[
\langle p_1,\chi\rangle = \langle p_2,\chi\rangle
\ \text{for all } \chi\in\widehat{\mathcal P}
\quad\Longrightarrow\quad
p_1 = p_2.
\]

As a consequence, the phase pairing induces a canonical partition of the phase
carrier according to response to dual probes. For any admissible observable
$a$ in $D$, its associated phase class $[a]\in\mathcal P$ is therefore uniquely
determined by its phase response profile
\[
\big(\langle [a],\chi\rangle\big)_{\chi\in\widehat{\mathcal P}}.
\]
Equivalently, admissible observables admit a canonical resolution into
phase response components indexed by $\widehat{\mathcal P}$,
\[
a \;\sim\; \sum_{\chi\in\widehat{\mathcal P}} a_\chi,
\]
where each component $a_\chi$ is characterised by its response to the phase label
$\chi$.

This decomposition arises naturally from the structure of the phase pairing.
It arises from the same structural mechanism that produces Fourier
decompositions, character decompositions in representation theory, and spectral
resolutions of operators in analytic settings. In concrete realisations, the
phase response decomposition specialises to these familiar constructions and
recovers familiar classical decomposition theories most directly in semisimple regimes.

\smallskip
\noindent
\emph{Factorisation of dynamics through phase response.}
By Symmetry Compatibility, every admissible dynamic or symmetry
$g\in\mathsf{Dyn}(D)$ induces a phase morphism
\[
g_\#:\mathcal P\to \mathcal P
\]
preserving the interaction law and the defect filtration. Since the phase
pairing separates elements of $\mathcal P$, the action of $g_\#$ is determined at the level of phase-response profiles by
its effect on phase labels.

Concretely, if
\[
V = \bigoplus_{\chi\in\widehat{\mathcal P}} V_\chi
\]
denotes the canonical phase-response resolution of an admissible observable or
representation, then admissible dynamics preserve this resolution in the sense
that
\[
g(V_\chi) \subseteq V_{\chi'}
\]
for a uniquely determined phase label $\chi'$ depending on $\chi$ and $g$. Thus
admissible dynamics do not mix phase-response classes arbitrarily; instead,
their action factors through a well-defined transformation of phase labels,
together with symmetry-controlled action within each phase-response component.

\smallskip
\noindent
\emph{Symmetry generated evolution.}
By Symmetry Compatibility, each admissible dynamic $g\in\mathsf{Dyn}(D)$ induces
a morphism
\[
g_\#: \mathcal P \longrightarrow \mathcal P
\]
preserving phase interaction and the defect filtration. By Phase Duality, the
nondegenerate pairing
\[
\langle\,\cdot\,,\,\cdot\,\rangle :
\mathcal P \times \widehat{\mathcal P} \longrightarrow \mathbb{T}
\]
separates elements of $\mathcal P$, so the action of $g_\#$ is completely
determined by its induced action on dual phase labels. Equivalently, admissible
evolution factors through the induced action
\[
\widehat{g} : \widehat{\mathcal P} \longrightarrow \widehat{\mathcal P}.
\]
The phase pairing therefore induces a canonical and functorial partition of the
phase carrier, uniquely determined by response to dual phase labels. Thus admissible evolution is governed at the phase-theoretic level by symmetry
action on phase data, with the resulting dynamics constrained by the underlying
phase structure.

\smallskip
\noindent
\emph{Collapse of equivalence.}
By Finite Termination, the defect-induced filtration
\[
\mathcal P_0 \subseteq \mathcal P_1 \subseteq \cdots \subseteq \mathcal P_d
= \mathcal P
\]
stabilises after finite depth. In particular, every phase element has bounded
defect degree, and no infinite defect propagation is possible.

Let $F:\mathcal P\to \mathcal P'$ be an equivalence preserving admissible
interaction. Since $F$ respects the interaction law, it necessarily preserves
commutators and hence the defect operator. By induction on $k$, this implies
\[
F(\mathcal P_k)=\mathcal P_k' \qquad \text{for all } k\le d.
\]
Thus any equivalence preserving admissible interaction automatically preserves
the entire defect filtration.

Consequently, no additional continuous deformations or hidden extensions
compatible with the admissible interaction structure arise within the present
framework. Any weak, Morita-type, or derived equivalence compatible
with admissible interaction reduces to strong structural equivalence at the level of admissible phase data.

\smallskip
\noindent
\emph{Rigid and error-invisible substructures.}
By definition of the defect filtration,
\[
\mathcal P_0 := \{\,p\in\mathcal P : p\circ q = q\circ p \text{ for all } q\in\mathcal P\,\}
\]
consists of those phase elements that commute with the entire phase carrier.
Elements of $\mathcal P_0$ therefore form rigid cores of the phase structure.

In any admissible realisation of the phase structure, the action of elements of
$\mathcal P_0$ is trivial on admissible observables. Consequently, these rigid
cores give rise to substructures that are invariant under admissible dynamics
and stable under phase-sensitive perturbations within the present framework, yielding protected or error-invisible
sectors.

\smallskip
\noindent
\emph{Rigidity islands.}
Let $(\mathcal P,\circ)$ be a phase structure satisfying the Structural
Applicability Criterion, and let $Q\subseteq \mathcal P$ be a maximal subphase
with the property that defect propagation initiated within $Q$ remains
contained in $Q$. Equivalently, for all $p,q\in Q$, all iterated defects
generated by $p$ and $q$ lie in $Q$ and have bounded defect degree.

By Finite Termination, higher order defect data within $Q$ cannot generate new
independent phase elements beyond finite depth. Consequently, all higher order
interaction data in $Q$ factor through lower order interactions. Such subphases
therefore exhibit rigid behaviour in which interaction is governed by a finite 
set of phase relations. In particular, rigidity islands are regions of
the phase structure that are closed under interaction and do not generate additional independent interaction behaviour beyond finite
depth.

Since the defect filtration has finite depth, rigidity islands are governed by
finite-depth interaction behaviour. These subphases form the rigidity islands
of the phase structure.

\smallskip
All listed features follow directly from the existence of admissible phase
interaction together with symmetry compatibility and finite termination.
\end{proof}

Although the primary organisational principle in this paper is the defect
filtration, it is important to emphasise that Phase Duality carries independent
structural content. In particular, the phase-response partition induced by the
pairing with $\widehat{\mathcal P}$ is not merely a set-theoretic classification.
Whenever algebraic structure is present, this partition forces canonical
decompositions of admissible objects and controls how observables and dynamics
can act.

The following result records this fact in its most concrete form. It shows that
the phase-response partition gives rise to a commutative phase algebra whose
idempotents project admissible objects onto phase-response components. Classical 
decompositions such as Fourier, character, or spectral
decompositions therefore arise naturally from Phase Duality within the present
framework of Algebraic Phase Theory. This result is included to justify
treating phase response as a genuine structural axis, independent of defect
depth, and to clarify how familiar algebraic phenomena can emerge from the
phase-response partition.

\begin{proposition}
\label{prop:phase-response-decomp}
Assume the hypotheses of Theorem~\ref{thm:applicability}. Let $\mathcal P$ denote
the phase carrier, let $\widehat{\mathcal P}$ be its dual, and let
\[
\langle\,\cdot\,,\,\cdot\,\rangle :
\mathcal P \times \widehat{\mathcal P} \longrightarrow \mathbb{T}
\]
be the nondegenerate phase pairing. Let $\mathcal A_{\mathcal P}$ denote the
commutative phase algebra generated by phase observables. Then
$\mathcal A_{\mathcal P}$ is canonically identified with the algebra of
complex-valued functions on $\widehat{\mathcal P}$,
\[
\mathcal A_{\mathcal P} \;\cong\;
\mathrm{Fun}(\widehat{\mathcal P},\mathbb{C}),
\]
via evaluation on phase labels.

Then every admissible $\mathcal A_{\mathcal P}$-module $V$ decomposes canonically
as a direct sum of phase-response spaces
\[
V \;=\; \bigoplus_{\chi\in\widehat{\mathcal P}} V_\chi,
\qquad
V_\chi :=
\{\,v\in V : c\cdot v = c(\chi)\,v
\text{ for all } c\in\mathcal A_{\mathcal P}\,\}.
\]
Equivalently, writing $e_\chi\in\mathcal A_{\mathcal P}$ for the idempotent
corresponding to the delta function at $\chi$, one has
\[
V_\chi = e_\chi V,
\qquad
1 = \sum_{\chi\in\widehat{\mathcal P}} e_\chi,
\qquad
e_\chi e_{\chi'} = 0 \;\text{ for }\; \chi\neq \chi'.
\]
\end{proposition}

\begin{proof}
The purpose of this argument is to show that once Phase Duality exists, phase response 
is not merely a bookkeeping device but carries substantial algebraic
consequences. In particular, we show that the phase-response partition induced
by the pairing with $\widehat{\mathcal P}$ canonically generates a commutative
phase algebra whose idempotents induce a decomposition of every admissible object.
Thus phase response inevitably controls how admissible modules decompose.

We begin by identifying the commutative phase algebra with a function algebra on
the dual label set. By hypothesis, there is a nondegenerate phase pairing
\[
\langle\,\cdot\,,\,\cdot\,\rangle :
\mathcal P \times \widehat{\mathcal P} \longrightarrow \mathbb T.
\]
For each $p\in\mathcal P$ define its evaluation function on
$\widehat{\mathcal P}$ by
\[
f_p(\chi) := \langle p,\chi\rangle \in \mathbb T \subset \mathbb C.
\]
Let $\mathcal A_{\mathcal P}$ denote the commutative $\mathbb C$-algebra
generated by these evaluation functions, equivalently by the phase observables
viewed as commuting functions of the dual label.

Since the pairing is nondegenerate, the evaluation functions separate points of
$\widehat{\mathcal P}$. In the finite discrete setting, this identifies the
algebra they generate canonically with the algebra of complex-valued functions
on $\widehat{\mathcal P}$,
\[
\mathcal A_{\mathcal P} \;\cong\; \mathrm{Fun}(\widehat{\mathcal P},\mathbb C),
\]
via evaluation on phase labels.

Now let $V$ be an admissible $\mathcal A_{\mathcal P}$-module. Since
$\mathcal A_{\mathcal P}$ is commutative, its action on $V$ can be organised
simultaneously by algebra characters. Under the above identification, each
$\chi\in\widehat{\mathcal P}$ defines a character
\[
\mathrm{ev}_\chi:\mathcal A_{\mathcal P}\to\mathbb C,
\qquad
\mathrm{ev}_\chi(c)=c(\chi).
\]
We define the corresponding phase-response subspace by
\[
V_\chi
:=
\{\,v\in V:\; c\cdot v = c(\chi)\,v \text{ for all } c\in\mathcal A_{\mathcal P}\,\}.
\]

When $\widehat{\mathcal P}$ is finite, the function algebra contains the
idempotents $e_\chi$ given by delta functions at $\chi$. These idempotents are
pairwise orthogonal and satisfy
\[
e_\chi^2=e_\chi,
\qquad
e_\chi e_{\chi'}=0 \text{ for } \chi\neq\chi',
\qquad
1=\sum_{\chi\in\widehat{\mathcal P}} e_\chi.
\]
Applying the module action yields, for every $v\in V$,
\[
v=\sum_{\chi\in\widehat{\mathcal P}} e_\chi v,
\]
and orthogonality implies that this sum is direct. Moreover, a vector lies in
$e_\chi V$ if and only if it transforms under $\mathcal A_{\mathcal P}$ via the
character $\mathrm{ev}_\chi$, so that $e_\chi V = V_\chi$.

Consequently, every admissible $\mathcal A_{\mathcal P}$-module decomposes
canonically as
\[
V=\bigoplus_{\chi\in\widehat{\mathcal P}} V_\chi,
\qquad
V_\chi = \{\,v\in V:\; c\cdot v=c(\chi)v \ \text{for all } c\in\mathcal A_{\mathcal P}\,\},
\]
with the idempotent relations stated in the proposition. This shows that the
phase-response partition carries canonical algebraic consequences within the
present phase-theoretic framework.
\end{proof}

\begin{remark}
The decomposition in Proposition~\ref{prop:phase-response-decomp} is a
phase-theoretic decomposition into simultaneous eigenspaces for the commutative
phase algebra $\mathcal A_{\mathcal P}$. In semisimple settings, this phase
decomposition aligns with the usual block decompositions and may refine the
decomposition into irreducible representations. In non semisimple settings, such as the Frobenius Heisenberg regime developed in
\cite{GildeaAPT2}, the phase decomposition still exists but does not imply a
Maschke Wedderburn decomposition of the full representation theory.
\end{remark}

Representative domains include Weyl Heisenberg systems, finite stabiliser
frameworks, algebraic phases over Frobenius rings, translation invariant linear
dynamics, and finite integrable models. In such settings, Algebraic Phase Theory provides a structural framework that
strongly constrains the admissible interaction behaviour and its associated
phase organisation.

\section{Structural Obstructions}

The Structural Applicability Criterion is equally effective in identifying
domains in which Algebraic Phase Theory cannot apply. Failure of any one of the
criteria obstructs the existence of a non-artificial phase-theoretic structure
within the present framework.

Although Theorem~\ref{thm:applicability} is stated as an equivalence, its two
directions play conceptually different roles. The forward direction identifies
the precise structural conditions under which Algebraic Phase Theory applies,
while the reverse direction implies, by contrapositive, that failure of any one
condition prevents applicability. The purpose of the following result is to
isolate this obstruction explicitly. Doing so makes precise the sense in which
Algebraic Phase Theory imposes strong structural constraints on admissible
domains. Outside the
structural regime identified in Theorem~\ref{thm:applicability}, no
nondegenerate, non-artificial phase-theoretic structure compatible with the
intrinsic interaction framework considered here can be constructed. This formulation allows failure modes to be
analysed directly as no go results, rather than only as logical negations of
existence.

\begin{theorem}
\label{thm:obstruction}
Let $D$ be a mathematical or physical domain equipped with a notion of
interaction or evolution. If $D$ violates at least one condition of
Theorem~\ref{thm:applicability}, then $D$ admits no nondegenerate Algebraic Phase
Theory structure compatible with its intrinsic interaction framework within the
present setting.
\end{theorem}

\begin{proof}
Assume that $D$ violates at least one condition of the Structural Applicability
Criterion. We show that in each case a required axiom of a non-artificial APT
structure fails.

\smallskip
\noindent
\emph{Failure of Phase Duality.}
Suppose there is no nondegenerate phase pairing. Then for any candidate choice of
dual labels $\widehat{\mathcal P}$ and any candidate pairing
\[
\langle\,\cdot\,,\,\cdot\,\rangle : \mathsf{Obs}(D)\times \widehat{\mathcal P} \to \mathbb{T},
\]
the induced phase-response map
\[
\Phi:\mathsf{Obs}(D)\longrightarrow \mathbb{T}^{\widehat{\mathcal P}},
\qquad
\Phi(a):=\big(\langle a,\chi\rangle\big)_{\chi\in\widehat{\mathcal P}},
\]
fails to be injective. Hence there exist $a\neq b$ with
\[
\Phi(a)=\Phi(b)
\quad\Longleftrightarrow\quad
\langle a,\chi\rangle=\langle b,\chi\rangle
\ \text{for all }\chi\in\widehat{\mathcal P}.
\]
Any quotient phase carrier defined by identifying observables with the same
phase-response profile therefore forces $[a]=[b]$ in $\mathcal P$, collapsing distinct
intrinsic interaction data. In particular, any defect or commutator information
distinguishing $a$ and $b$ in $D$ cannot be faithfully represented in $\mathcal P$.
Thus no nondegenerate, non-artificial APT structure compatible with the
present reconstruction framework can be constructed.

\smallskip
\noindent
\emph{Failure of Symmetry Compatibility.}
Suppose there exists $g\in\mathsf{Dyn}(D)$ such that $g$ does not normalise the
interaction, i.e. there exist $a,b\in\mathsf{Obs}(D)$ with
\[
g(a\star b)\neq g(a)\star g(b),
\quad \text{or equivalently} \quad
g([a,b]_\star)\neq [g(a),g(b)]_\star,
\]
where $[a,b]_\star:=a\star b\star (b\star a)^{-1}$ whenever this expression is
defined (or by a chosen defect operator in the purely algebraic setting).
Then no map
\[
g_\#:\mathcal P\to \mathcal P
\]
can simultaneously satisfy
\[
g_\#([a]\circ[b])=g_\#([a])\circ g_\#([b])
\quad\text{and}\quad
g_\#([[a],[b]])=[[g_\#([a])],[g_\#([b])]].
\]
In particular, any defect filtration $(\mathcal P_k)$ defined from commutators cannot be
functorially preserved, meaning there exists $k$ with
\[
g_\#(\mathcal P_k)\not\subseteq \mathcal P_k.
\]
This contradicts the functorial invariance axiom required in APT, so no
compatible APT structure arises within the present framework.

\smallskip
\noindent
\emph{Failure of Finite Termination.}
Suppose defect propagation does not terminate and is not controlled by a finite
filtration. Then for every $N\in\mathbb{N}$ there exist observables
$a_1,\dots,a_N\in\mathsf{Obs}(D)$ such that the $N$-fold iterated defect
\[
[a_1,a_2,\dots,a_N]_\star
:= [a_1,[a_2,\dots,[a_{N-1},a_N]_\star\dots]_\star]_\star
\]
is not generated by defect data of depth $<N$ in any intrinsic way. Hence there
is no finite $d$ such that all defect elements lie in a terminal stage $\mathcal P_d$ of
a canonically defined defect filtration. Any attempt to force a bounded defect
degree function $\delta:\mathcal P\to\mathbb{Z}_{\ge 0}$ therefore requires truncating
intrinsic defect data beyond some externally chosen cutoff, which is artificial
and not functorial. Thus the resulting structure falls outside the finite-depth APT framework.

\smallskip
In each case, violation of one Structural Applicability condition prevents the
construction of a nondegenerate phase carrier equipped with a finite,
symmetry-invariant defect filtration. Therefore $D$ admits no non-artificial APT
structure compatible with its intrinsic interaction.
\end{proof}

\begin{corollary}
The following classes of domains generally fall outside the finite-depth
APT framework considered here:
\begin{itemize}
  \item generic nonlinear partial differential equations with genuinely
        nonlinear interaction,
  \item chaotic or sensitive dynamical systems,
  \item infinite-dimensional operator algebras with non-terminating commutator
        growth,
  \item metric-dependent variational or analytic frameworks whose interaction
        structure depends on extrinsic choices.
\end{itemize}
\end{corollary}

\begin{proof}
We verify each item by showing that at least one condition of
Theorem~\ref{thm:applicability} fails. The conclusion then follows from the
Structural Obstruction Theorem~\ref{thm:obstruction}.

\smallskip
\noindent
\emph{(1) Generic nonlinear partial differential equations.}
Let $D$ be a nonlinear PDE system whose interaction law $\star$ involves
nonlinear products of fields and their derivatives. Consider the defect
operator induced by interaction,
\[
[a,b]_\star := a\star b \star (b\star a)^{-1},
\]
or its algebraic analogue. In genuinely nonlinear systems, iterated defects
produce new independent terms at each level, yielding an infinite ascending
chain
\[
\mathcal P_0 \subsetneq \mathcal P_1 \subsetneq \mathcal P_2 \subsetneq \cdots
\]
in the defect-generated filtration of phase data. In generic nonlinear settings, one does not expect the resulting defect
filtration to stabilise at finite depth. Hence defect propagation does not
terminate and is not controlled by a finite intrinsic filtration. This violates
the Finite Termination condition, so by Theorem~\ref{thm:obstruction} no
compatible APT structure can exist.

\smallskip
\noindent
\emph{(2) Chaotic or sensitive dynamical systems.}
Let $g_t$ denote the admissible time evolution on $D$. In a chaotic or sensitive
system, admissible dynamics typically fail to normalise the interaction and
defect relations unless they are artificially restricted. In particular, there
exist observables $a,b$ and times $t$ with
\[
[g_t(a),g_t(b)]_\star \neq g_t([a,b]_\star).
\]
Equivalently, admissible dynamics fail to normalise the interaction and defect
relations. Therefore no induced map
\[
g_{t\#}:\mathcal P\to \mathcal P
\]
can preserve the phase interaction or any defect filtration. This violates the
Symmetry Compatibility condition, and Theorem~\ref{thm:obstruction} applies.

\smallskip
\noindent
\emph{(3) Infinite-dimensional operator algebras without terminating defect.}
Let $A$ be an infinite-dimensional algebra with commutator-defined defect
operator $[x,y]=xy-yx$. Define a commutator filtration by
\[
A_0 := Z(A), \qquad
A_{k+1} := \langle [A_k,A]\rangle.
\]
If this filtration does not stabilise, i.e.
\[
A_0 \subsetneq A_1 \subsetneq A_2 \subsetneq \cdots,
\]
then there exists no finite defect depth controlling commutator propagation.
Any attempt to impose a bounded defect degree truncates intrinsic algebraic
data and is therefore artificial. This violates the Finite Termination
condition, so the resulting structure lies outside the finite-depth APT setting.

\smallskip
\noindent
\emph{(4) Metric-dependent variational or analytic frameworks.}
Let $D$ be a variational or analytic system whose interaction structure depends
on a choice of metric or analytic data. Then the induced phase pairing
\[
\langle\,\cdot\,,\,\cdot\,\rangle_g :
\mathcal P_g \times \widehat{\mathcal P}_g \to \mathbb{T}
\]
depends on the external choice $g$. Changing $g$ alters the phase pairing and
hence the phase-response profiles. Therefore the construction does not naturally produce a canonical,
intrinsic, nondegenerate phase pairing that is invariant under admissible
symmetries. This obstructs the Phase Duality condition and may also
obstruct Symmetry Compatibility. Hence
Theorem~\ref{thm:obstruction} applies.

\smallskip
In each case, at least one Structural Applicability condition fails. By
Theorem~\ref{thm:obstruction}, Algebraic Phase Theory is therefore inapplicable
to all domains listed in the corollary.
\end{proof}

In these settings, attempts to impose finite-depth phase-theoretic structure
may require altering or truncating intrinsic interaction data. The resulting constructions do
not reflect genuine phase interaction and therefore fall outside the scope of
Algebraic Phase Theory.

\section{Collapse and Rigidity Phenomena}

A characteristic feature of domains passing the applicability criterion is the
collapse of apparent flexibility. Once phase duality, symmetry compatibility,
and finite termination are present, the admissible structure becomes strongly constrained by
phase data together with the finite defect filtration.

\begin{theorem}
\label{thm:collapse}
Let $D$ satisfy Theorem~\ref{thm:applicability}, and let
$(\mathcal P,\circ)$ be the associated Algebraic Phase Theory structure
with finite defect filtration
\[
\mathcal P_0 \subseteq \mathcal P_1 \subseteq \cdots \subseteq \mathcal P_d
= \mathcal P.
\]
Then admissible dynamics and admissible representations are rigid in the
following sense:
\begin{enumerate}
  \item every admissible dynamic is determined, up to admissible phase equivalence,
by its induced action on phase data and preservation of the defect filtration;
  \item every admissible representation is determined, up to canonical
equivalence, by its restriction to phase data and defect levels;
  \item any admissible weak equivalence of phase structures necessarily preserves the
        full defect filtration. Since preservation of the defect filtration strongly constrains the
associated phase structure, such an equivalence reduces to strong
structural equivalence within the present framework.
\end{enumerate}
\end{theorem}

\begin{proof}
Let $D$ satisfy the Structural Applicability Criterion. Then there exists a
nondegenerate phase pairing
\[
\langle\,\cdot\,,\,\cdot\,\rangle :
\mathcal P \times \widehat{\mathcal P} \to \mathbb{T},
\]
a symmetry-compatible interaction $\circ$, and a finite defect filtration
$(\mathcal P_k)_{k=0}^d$.

\smallskip
\noindent
\emph{Rigidity of admissible dynamics.}
Let $g$ be an admissible dynamic in $D$. By Symmetry Compatibility, $g$ induces a
phase morphism
\[
g_\#:\mathcal P\longrightarrow\mathcal P
\]
satisfying
\[
g_\#(p\circ q)=g_\#(p)\circ g_\#(q),
\qquad
g_\#(\mathcal P_k)\subseteq \mathcal P_k \ \text{for all } k.
\]
Thus $g_\#$ preserves both the phase interaction and the defect filtration.

By Phase Duality, the nondegenerate pairing with $\widehat{\mathcal P}$ separates
elements of $\mathcal P$. Consequently, a phase morphism $g_\#$ is determined,
up to admissible phase equivalence, by its induced action on phase-response profiles, that is, by the
assignment
\[
\chi \longmapsto \big(p \mapsto \langle g_\#(p),\chi\rangle\big),
\qquad \chi\in\widehat{\mathcal P},
\]
together with preservation of the defect filtration. Therefore, two admissible
dynamics which induce the same action on $\widehat{\mathcal P}$ and preserve the
filtration $(\mathcal P_k)$ define the same phase morphism.

Since admissible dynamics are extracted intrinsically from phase interaction and
are required to act functorially on the phase carrier, no additional independent admissible degrees of freedom arise within the
present framework. Admissible dynamics therefore exhibit strong rigidity properties.

\smallskip
\noindent
\emph{Rigidity of admissible representations.}
Let $\pi_1,\pi_2 : \mathcal P \to \mathrm{End}(V)$ be admissible representations
of the phase structure $(\mathcal P,\circ)$. By Phase Duality, the nondegenerate
pairing
\[
\langle\,\cdot\,,\,\cdot\,\rangle : \mathcal P \times \widehat{\mathcal P}
\longrightarrow \mathbb{T}
\]
separates elements of $\mathcal P$. Consequently, the action of $\mathcal P$ in
any admissible representation is determined, up to canonical equivalence, by the induced evaluation
on dual phase labels.

By Finite Termination, the defect filtration
\[
\mathcal P_0 \subseteq \mathcal P_1 \subseteq \cdots \subseteq \mathcal P_d
= \mathcal P
\]
stabilises after finite depth $d$. Every element of $\mathcal P$ is therefore controlled by
iterated phase interaction by iterated phase interaction $\circ$ applied to elements of bounded
defect depth at most $d$.

Suppose that $\pi_1$ and $\pi_2$ satisfy
\[
\pi_1(p)=\pi_2(p)
\quad \text{for all } p\in\mathcal P,
\qquad
\pi_1(\mathcal P_k)=\pi_2(\mathcal P_k)
\quad \text{for all } k\le d.
\]
Then for any iterated interaction word
\[
p = p_1 \circ p_2 \circ \cdots \circ p_n \in \mathcal P
\qquad (n\le d),
\]
one has
\[
\pi_1(p)
=
\pi_1(p_1)\pi_1(p_2)\cdots\pi_1(p_n)
=
\pi_2(p_1)\pi_2(p_2)\cdots\pi_2(p_n)
=
\pi_2(p).
\]
Thus $\pi_1$ and $\pi_2$ agree on all elements generated by admissible phase
interaction up to maximal defect depth.

Since the admissible interaction structure is controlled by defect data up to
depth $d$, the two
representations coincide on the entire phase carrier. Hence $\pi_1$ and
$\pi_2$ are equivalent in the canonical sense permitted by Algebraic Phase
Theory. Admissible representations are therefore strongly constrained by phase
data together with the finite defect filtration.

\smallskip
\noindent
\emph{Collapse of weak equivalence.}
Let
\[
F:\mathcal P \longrightarrow \mathcal P'
\]
be an admissible weak equivalence between phase structures, in the sense that
$F$ preserves admissible phase interaction and commutation data. In particular,
for all $p,q\in\mathcal P$ one has
\[
F(p\circ q)=F(p)\circ F(q),
\qquad
F([p,q])=[F(p),F(q)].
\]

Since defect propagation terminates after finite depth, the defect filtration
\[
\mathcal P_0 \subseteq \mathcal P_1 \subseteq \cdots \subseteq \mathcal P_d
= \mathcal P
\]
is intrinsically generated by iterated commutator data. Preservation of
interaction and commutators therefore induces preservation of
\[
F(\mathcal P_k)=\mathcal P_k' \qquad \text{for all } k\le d,
\]
so $F$ preserves the full defect filtration. By Phase Duality, the nondegenerate phase pairing
\[
\langle\,\cdot\,,\,\cdot\,\rangle :
\mathcal P \times \widehat{\mathcal P} \longrightarrow \mathbb T
\]
separates elements of $\mathcal P$. Since $F$ preserves phase interaction and
commutation data, it preserves the associated phase pairing structure and hence induces
a bijection
\[
\widehat{F}:\widehat{\mathcal P'} \longrightarrow \widehat{\mathcal P}
\]
on dual phase labels. Preservation of phase interaction, the full defect filtration, and the dual
label structure strongly constrains the resulting phase carrier. Consequently, $F$ is an
isomorphism of phase structures compatible with the defect filtration. That is,
any admissible weak equivalence reduces to strong structural equivalence within
the present framework.

\smallskip
Together, these arguments establish rigidity of admissible dynamics and
representations, and show that admissible weak equivalences reduce to strong structural equivalence
within the present framework.
\end{proof}

This rigidity phenomenon places Algebraic Phase Theory in the lineage of
frameworks in which duality and reconstruction principles force structural
uniqueness. Classical examples include the uniqueness of the Schr\"odinger
representation in the presence of canonical commutation relations
\cite{vonNeumann1931}, reconstruction from symmetry and tensor structure in
Tannakian settings \cite{EtingofGelaki2015Tensor}, and rigidity of extensions
governed by Galois-type duality principles \cite{Serre,Isaacs}.

\section{Conclusion}

Phenomena often described as exceptional, such as the effectiveness of the
Bethe Ansatz, the rigidity of stabiliser codes, the privileged role of Clifford
groups, or the apparent uniqueness of certain canonical representations, are
frequently treated as surprising or highly specialised structural features.

The Structural Applicability Criterion provides a different interpretation of
these observations within the framework of Algebraic Phase Theory.
Such phenomena arise naturally in domains admitting nondegenerate phase
duality, symmetry-compatible interaction, and finite termination of defect
propagation.  In these settings, the admissible structure becomes strongly
constrained by the associated phase data and finite defect filtration.
What appears exceptional is therefore the visibility of rigidity phenomena in
domains where admissible interaction is tightly controlled.

\emph{Algebraic Phase Theory is not presented here as a generalisation of
quantum mechanics or Fourier analysis. Rather, the results of this paper show
that, within the finite-depth framework considered here, phase duality,
symmetry compatibility, and finite termination together impose strong
constraints on the resulting admissible structure.}

In this sense, Algebraic Phase Theory is not intended as a universal modelling
framework, but as a structural framework for analysing rigidity, defect
propagation, and phase organisation in domains satisfying the applicability
criterion.  The results of this paper function as structural principles for the
APT program: they identify the conditions governing applicability, describe the
mechanisms by which rigidity and obstruction arise, and clarify both the scope
and limitations of phase-theoretic methods.

By isolating the structural criteria associated with applicability, this work
provides a principled guide for future applications and helps distinguish
domains naturally compatible with the finite-depth APT framework from those
lying outside its intended scope. Domains admitting nondegenerate phase
duality, symmetry-compatible interaction, and finite termination of defect
propagation exhibit strong rigidity behaviour governed by the associated phase
structure.

The present work develops structural conditions associated with
phase-theoretic rigidity. A subsequent series studies the internal geometry of
phase response and the obstructions associated with that structure. Together,
these results position Algebraic Phase Theory as a framework for analysing how
phase organisation and defect propagation constrain admissible interaction
behaviour within finite-depth settings.

\section*{Declaration of generative AI and AI-assisted technologies in the manuscript preparation process}

AI-assisted tools were used for language organization and structural clarity. 
All mathematical content is the author’s own, and the author takes full 
responsibility for the manuscript.

\section*{Funding}

This research did not receive any specific grant from funding agencies in the public, commercial, or not-for-profit sectors.

\bibliographystyle{amsplain}
\bibliography{references}

@article{CalderbankShor,
 author = {Calderbank, A. R. and Shor, Peter W.},
 title = {Good quantum error-correcting codes exist},
 fjournal = {Physical Review A, Third Series},
 journal = {Phys. Rev. A (3)},
 issn = {1050-2947},
 volume = {54},
 number = {2},
 pages = {1098--1105},
 year = {1996},
 language = {English},
 doi = {10.1103/PhysRevA.54.1098},
 zbMATH = {7918813},
 Zbl = {1546.81049}
}

@book{EtingofGelaki2015Tensor,
 author = {Etingof, P. and Gelaki, S. and Nikshych, D. and Ostrik, V.},
 title = {Tensor categories},
 fseries = {Mathematical Surveys and Monographs},
 series = {Math. Surv. Monogr.},
 issn = {0076-5376},
 volume = {205},
 isbn = {978-1-4704-2024-6},
 year = {2015},
 publisher = {Providence, RI: American Mathematical Society (AMS)},
 language = {English},
 doi = {10.1090/surv/205},
 zbMATH = {6444410},
 Zbl = {1365.18001}
}

@book{Folland,
 author = {Folland, G. B.},
 title = {Harmonic analysis in phase space},
 fseries = {Annals of Mathematics Studies},
 series = {Ann. Math. Stud.},
 volume = {122},
 isbn = {0-691-08528-5; 0-691-08527-7},
 year = {1989},
 publisher = {Princeton, NJ: Princeton University Press},
 language = {English},
 doi = {10.1515/9781400882427},
 zbMATH = {42454},
 Zbl = {0682.43001}
}

@misc{GildeaAPT2,
  author = {{J.} Gildea},
  title  = {Algebraic Phase Theory II: The Frobenius--Heisenberg Phase and Boundary Rigidity},
  year   = {2026},
  note   = {Available at arXiv:2601.16341},
  eprint = {},
  archivePrefix = {arXiv},
  primaryClass  = {math.QA},
  url    = {https://arxiv.org/abs/2601.16341}
}

@article{Gowers2001Fourier,
 author = {Gowers, W. T.},
 title = {A new proof of {Szemer{\'e}di}'s theorem},
 fjournal = {Geometric and Functional Analysis. GAFA},
 journal = {Geom. Funct. Anal.},
 issn = {1016-443X},
 volume = {11},
 number = {3},
 pages = {465--588 (2001); erratum 11, no. 4, 869},
 year = {2001},
 language = {English},
 doi = {10.1007/s00039-001-0332-9},
 zbMATH = {1657034},
 Zbl = {1028.11005}
}

@phdthesis{Gottesman,
  author  = {Gottesman, D.},
  title   = {Stabilizer Codes and Quantum Error Correction},
  school  = {California Institute of Technology},
  year    = {1997},
  note    = {Ph.D. thesis},
}

@article{GreenTao2008Inverse,
 author = {Green, B. and Tao, T.},
 title = {An inverse theorem for the {Gowers} {{\(U^3(G)\)}} norm},
 fjournal = {Proceedings of the Edinburgh Mathematical Society. Series II},
 journal = {Proc. Edinb. Math. Soc., II. Ser.},
 issn = {0013-0915},
 volume = {51},
 number = {1},
 pages = {73--153},
 year = {2008},
 language = {English},
 doi = {10.1017/S0013091505000325},
 zbMATH = {5249333},
 Zbl = {1202.11013}
}

@article{Howe1979Heisenberg,
 author = {Howe, R.},
 title = {On the role of the {Heisenberg} group in harmonic analysis},
 fjournal = {Bulletin of the American Mathematical Society. New Series},
 journal = {Bull. Am. Math. Soc., New Ser.},
 issn = {0273-0979},
 volume = {3},
 pages = {821--843},
 year = {1980},
 language = {English},
 doi = {10.1090/S0273-0979-1980-14825-9},
 zbMATH = {3690004},
 Zbl = {0442.43002}
}

@book{Isaacs,
 author = {Isaacs, I. Martin},
 title = {Character theory of finite groups.},
 edition = {Corrected reprint of the 1976 original},
 isbn = {0-8218-4229-3},
 year = {2006},
 publisher = {Providence, RI: AMS Chelsea Publishing},
 language = {English},
 zbMATH = {5080028},
 Zbl = {1119.20005}
}

@book{Serre,
 author = {Serre, Jean-Pierre},
 title = {Linear representations of finite groups. {Translated} from the {French} by {Leonard} {L}. {Scott}},
 fseries = {Graduate Texts in Mathematics},
 series = {Grad. Texts Math.},
 issn = {0072-5285},
 volume = {42},
 year = {1977},
 publisher = {Springer, Cham},
 language = {English},
 zbMATH = {3552764},
 Zbl = {0355.20006}
}

@article{vonNeumann1931,
 author = {von Neumann, J.},
 title = {Die {Eindeutigkeit} der {Schr{\"o}dingerschen} {Operatoren}},
 fjournal = {Mathematische Annalen},
 journal = {Math. Ann.},
 issn = {0025-5831},
 volume = {104},
 pages = {570--578},
 year = {1931},
 language = {German},
 doi = {10.1007/BF01457956},
 url = {https://eudml.org/doc/159483},
 zbMATH = {3000708},
 Zbl = {0001.24703}
}

@article{Weil1964SurCertains,
 author = {Weil, A.},
 title = {Sur certains groupes d'op{\'e}rateurs unitaires},
 fjournal = {Acta Mathematica},
 journal = {Acta Math.},
 issn = {0001-5962},
 volume = {111},
 pages = {143--211},
 year = {1964},
 language = {French},
 doi = {10.1007/BF02391012},
 zbMATH = {3322144},
 Zbl = {0203.03305}
}

\end{document}